\scrollmode

\documentclass{amsart}
\usepackage{amssymb}
\usepackage[all]{xy}
\usepackage{hyperref}

\theoremstyle{plain}
\newtheorem{prop}{Proposition}
\newtheorem{thm}[prop]{Theorem}
\newtheorem{cor}[prop]{Corollary}

\newtheorem{fact}[prop]{Fact}
\newtheorem{step}{Claim}

\newtheorem*{thmA}{Theorem A}
\newtheorem*{thmB}{Theorem B}
\newtheorem*{thmC}{Theorem C}

\usepackage{amsfonts}
\usepackage{amsmath}
\usepackage{amsthm}
\usepackage{textcomp}
\usepackage{amscd}

\theoremstyle{definition}

\theoremstyle{remark}

\newtheorem{rem}[prop]{Remark}
\newtheorem{example}[prop]{Example}

\numberwithin{prop}{section} 
\numberwithin{equation}{section}

\DeclareMathOperator{\cha}{char}

\DeclareMathOperator{\Endm}{End}
\DeclareMathOperator{\Zen}{Z}
\DeclareMathOperator{\image}{im}
\DeclareMathOperator{\kernel}{ker}

\DeclareMathOperator{\rank}{rk}
\DeclareMathOperator{\res}{res}
\DeclareMathOperator{\tra}{tr}
\DeclareMathOperator{\spa}{Span}
\DeclareMathOperator{\abel}{ab}

\DeclareMathOperator{\adj}{ad}

\newcommand{\opno}{\lhd_o}

\newcommand{\clno}{\lhd_c}
\newcommand{\clsgp}{\leq_c}

\newcommand{\Z}{\mathbb{Z}}
\newcommand{\Q}{\mathbb{Q}}
\newcommand{\F}{\mathbb{F}}

\newcommand{\dbl}{[\![}
\newcommand{\dbr}{]\!]}

\newcommand{\argu}{\hbox to 7truept{\hrulefill}}

\author{C. Quadrelli}
\title{Bloch-Kato pro-p groups and locally powerful groups}

\address{Dipartimento di Matematica, Universit\`a di Milano-Bicocca\\
Ed. U5, Via R.Cozzi 53, 20125 Milano, Italy}
\email{c.quadrelli1@campus.unimib.it}

\begin{document}
\begin{abstract}
A Bloch-Kato pro-$p$ group $G$ is a pro-$p$ group with the property
that the $\F_p$-cohomology ring of every closed subgroup of $G$ is quadratic.
It is shown that either such a pro-$p$ group $G$ contains no closed free pro-$p$ groups of infinite rank,
or there exists an {\it orientation} $\theta\colon G\rightarrow \Z_p^\times$
such that $G$ is $\theta$-abelian. (See Thm~B.) 
In case that $G$ is also finitely generated, this implies that 
$G$ is powerful, $p$-adic analytic with $d(G)=cd(G)$, and
its $\F_p$-cohomology ring is an exterior algebra (see Cor. \ref{corollmetabelian}).
These results will be obtained by studying locally powerful groups (see Thm~A).
There are certain Galois-theoretical implications,
since Bloch-Kato pro-$p$ groups arise naturally as maximal pro-$p$ quotients
and pro-$p$ Sylow subgroups of absolute Galois groups (see Corollary \ref{corollware}).
Finally, we study certain closure operations of the class of Bloch-Kato pro-$p$ groups, 
connected with the Elementary type conjecture.
\end{abstract}


\dedicatory{To Professor Helmut Koch, with admiration on his 80th birthday.}

\maketitle

\section{Introduction}\label{sec:intro}

Following \cite{bk} one calls a pro-$p$ group $G$ {\it a Bloch-Kato pro-$p$ group}
if the cohomology ring $H^\bullet(K,\F_p)$ is a quadratic
$\F_p$-algebra for every closed subgroup $K$ of $G$.
From the positive solution of the Bloch-Kato conjecture 
recently obtained by M.~Rost and V.~Voevodsky (with C.~Weibel's patch) one knows that for 
every field $F$ containing a primitive $p$th root of unity the maximal pro-$p$ quotient $G_F(p)$ of the absolute Galois group $G_F$ of $F$ is a Bloch-Kato pro-$p$ group
(see \cite{voevod1} and \cite{voevod2} for an overview of the proof, and \cite{mvw}, \cite{voevod3},
\cite{weibel}, \cite{weibel2} for the foundation and completion of the proof).

The main goal of this paper is to estabilish a strong version of Tits alternative for Bloch-Kato pro-$p$ groups.
(See \cite{tits} or  \cite{dixon} for the original Tits alternative on linear groups.)
For this purpose we study in section \ref{sec:locpow} {\it locally powerful} pro-$p$ groups $G$,
where we call a pro-$p$ group $G$ locally powerful if
every finitely generated closed subgroup $K$ of $G$ is powerful. 
In order to state the classification of torsion-free, finitely generated, locally powerful pro-$p$ groups
effectively, we will introduce the notion of an {\it oriented pro-$p$ group} $(G,\theta)$, i.e.,
$G$ is a pro-$p$  group and $\theta\colon G\to\Z_p^\times$ is a
(continuous) homomorphism of pro-$p$ groups,
where $\Z_p$ denotes the ring of $p$-adic integers, and $\Z_p^\times\subset\Z_p$ denotes its group of units.
For an oriented pro-$p$ group $(G,\theta)$ one has a
particular closed subgroup 
\begin{equation*}
\Zen_\theta(G)=\left\{\, h\in \ker(\theta)\left| ghg^{-1}=h^{\theta(g)}\right. \text{for all $g\in G$}\,\right\}
\end{equation*}
which will be called the {\it $\theta$-center} of $G$. The oriented pro-$p$ group $(G,\theta)$
will be called {\it $\theta$-abelian}, if $\Zen_\theta(G)=\ker(\theta)$. 
Thus every $\theta$-abelian pro-$p$ group is metabelian.
Obviously, if $\theta\equiv {\mathbf 1}$
is constant equal to $1$, $\Zen_{\mathbf 1}(G)$ coincides with the center of $G$, and 
$(G,{\mathbf 1})$ is ${\mathbf 1}$-abelian if,
and only if, $G$ is abelian. In \S \ref{subsec:proofA} we will prove the following theorem.

\begin{thmA}
A torsion-free finitely generated pro-$p$ group $G$ is locally powerful
if, and only if, there exists an orientation $\theta\colon G\rightarrow\Z_p^\times$
such that $(G,\theta)$ is $\theta$-abelian.
\end{thmA}

Using Theorem~A we will deduce in section \ref{sec:titsalt} the following Tits alternative-type
result for Bloch-Kato pro-$p$ groups (see Theorem \ref{thm3equivalenceBloch-KatoHP}).

\begin{thmB}
Let $p$ be an odd prime, and let $G$ be a Bloch-Kato pro-$p$ group. 
Then either $G$ is $\theta$-abelian for some orientation $\theta\colon G\to\Z_p^\times$
or $G$ contains a closed non-abelian free pro-$p$ subgroup.
\end{thmB}

For $p$ odd, a result similar to Theorem~B was already proved by R.~Ware for 
maximal pro-$p$ Galois groups
\cite[Theorem 1 and Corollary 1]{ware1}. His article was also a motivation for us
to look for a Tits alternative in the class of Bloch-Kato pro-$p$ groups.

In the final section we will consider free and direct products as well as inverse limits of Bloch-Kato
pro-$p$ groups. It will turn out that the class of Bloch-Kato pro-$p$ groups is closed 
under free pro-$p$ products (see Theorem \ref{freeprodofBloch-Kato}),
and under certain inverse limits (see Proposition \ref{propinverselimit}).
However, for direct products one has the following.

\begin{thmC}
Let $G_1$ and $G_2$ be Bloch-Kato pro-$p$ groups, and assume that $G_1\times G_2$ is Bloch-Kato as well.
Then the following restrictions hold:
\begin{itemize}
 \item[(i)] None of $G_1$ and $G_2$ is a powerful non-abelian Bloch-Kato group;
 \item[(ii)] at least one of the two groups is abelian.
\end{itemize}
Moreover, $\Z_p\times S$ is a Bloch-Kato pro-$p$ group for any free pro-$p$ group $S$.
\end{thmC}

The main reason for these last investigations is the connection
with the Elementary Type conjecture for maximal pro-$p$ Galois groups.

\section{Preliminaries}\label{sec:prelim}

We work in the category of pro-$p$ groups.
Henceforth subgroups are  to be considered closed
and all generators are to be considered topological generators (in the sense of the pro-$p$ topology).
For basic facts on Galois cohomology we refer to \cite{nsw} or \cite{gc}.
We abbreviate $H^k(G)$ for $H^k(G,\F_p)$ with the trivial $G$-action on $\F_p$.
Thus $H^\bullet(G)=\bigoplus_{k\geq0}H^k(G)$ denotes the graded cohomology ring equipped with the cup product $\cup$. 

The first Bockstein homomorphism $\beta: H^1(G)\rightarrow H^2(G)$ is the connecting homomorphism arising from the short exact
sequence of trivial $G$-modules \[0\longrightarrow \Z/p.\Z \longrightarrow \Z/p^2.\Z\longrightarrow \Z/p.\Z\longrightarrow 0.\]
When $p= 2$ one has $\beta(\chi) = \chi\cup\chi$ \cite[Lemma 2.4]{SunilIdoJan}.

If $G$ is finitely generated, we denote by $d(G)$ the minimal number of generators of $G$,
namely $d(G)=\dim(G/\Phi(G))$ as $\F_p$-vector space, where $\Phi(G)$ is the Frattini subgroup of $G$.
In particular, if $d=d(G)$, we say that $G$ is $d$-generated.
Moreover, the rank $\rank(G)$ of $G$ is $\sup\{d(K)|K\clsgp G\}$.
If $G=S/R$ is a minimal presentation for $G$, with $S$ a free pro-$p$ group such that $d(S)=d(G)$,
then the relation rank $r(G)$ is the minimal number of generators of $R$ as a closed normal subgroup of $S$.
Moreover, it is well known that 
\[d(G)=\dim_{\F_p}\left(H^1(G)\right)\quad \text{and} \quad r(G)=\dim_{\F_p}\left(H^2(G)\right)\]
(see \cite[Ch. III \S9]{nsw}).

Finally, ${^xy}=xyx^{-1}$, and $[x,y]={^xy}\cdot y^{-1}$ is the commutator of $x$ and $y$, for $x,y\in G$.

\bigskip

As mentioned in the introduction,
the maximal pro-$p$ Galois group $G_F(p)$ of a field $F$ containing the $p$th roots of unity $\mu_p$ is a Bloch-Kato group.
Indeed if $\cha F=p$ then $G_F(p)$ is a free pro-$p$ group (see \cite[II \S 2.2]{gc}),
which is Bloch-Kato since a free pro-$p$ group has cohomological dimension equal to 1.

Otherwise, for a profinite group $G$ let $\mathcal{O}^p(G)$ be the subgroup
\[\mathcal{O}^p(G)=\langle K\in\text{Syl}_\ell(G)|\ell\neq p\rangle,\]
where $\text{Syl}_\ell(G)$ is the set of the Sylow pro-$\ell$ subgroups; namely $G/\mathcal{O}^p(G)$ is the 
maximal pro-$p$ quotient of $G$ \cite[Proposition 2.1]{thomas}. 

Let $F(p)$ be the maximal $p$-extension of a field $F$ with $\cha F\neq p$.
Then the absolute Galois group of $F(p)$ is $G_{F(p)}=\mathcal{O}^p(G_F)$.
Since $F(p)$ satisfies the hypotesis of the Bloch-Kato conjecture
(i.e., $\cha(F(p))\neq p$ and $\mu_p\subseteq F(p)$),
also the cohomology ring $H^\bullet(\mathcal{O}^p(G_F))$ is quadratic.

Moreover, as $\mathcal{O}^p(G_F)$ is $p$-perfect, $H^1(\mathcal{O}^p(G_F))=0$.
Thus $H^\bullet(\mathcal{O}^p(G_F))=0$.
This implies that in the Lyndon-Hochschild-Serre spectral sequence
arising from $1\rightarrow\mathcal{O}^p(G_F)\rightarrow G_F\rightarrow G_F(p)\rightarrow1$ the terms
\[E_2^{rs}=H^r\left(G_F(p),H^s(\mathcal{O}^p(G_F))\right)\]
vanish for $s>0$, and the spectral sequence collapses at the $E_2$-term.
Hence the inflation map $H^n(G_F(p))\rightarrow H^n(G_F)$ is an isomorphism for every $n\geq0$ \cite[Lemma 2.1.2]{nsw}.
Thus $H^\bullet(G_F(p))$ is quadratic as $H^\bullet(G_F)$ is quadratic by the Bloch-Kato conjecture.

Note that all the $p$-Sylow subgroups of an absolute Galois group -- for any prime $p$ -- are Bloch-Kato pro-$p$ groups
(see \cite[\S 9]{SunilIdoJan}).

Bloch-Kato pro-$p$ groups have been defined and studied the first time in \cite{bk}.
A fundamental feature of Bloch-Kato groups is the following:
if $p$ is odd then a Bloch-Kato pro-$p$ group is \textit{torsion-free} \cite[Proposition 2.3]{bk},
whereas the only non-trivial finite (pro-)2 Bloch-Kato groups are the elementary abelian 2-groups \cite[Proposition 2.4]{bk}.

If we keep in mind the Galois-theoretical background,
this fact can be seen as an analogue of the celebrated Artin-Schreier theorem,
which states that the only non-trivial finite subgroup of an absolute Galois group is $C_2$.

\section{Locally powerful and oriented pro-$p$ groups}\label{sec:locpow}
\subsection{Powerful pro-$p$ groups and Lie algebras}

A pro-$p$ group $G$ is said to be {\it powerful} if
\[[G,G]\subseteq\left\{\begin{array}{cc}G^p & \text{for }p\text{ odd},\\G^4 & \text{for }p=2,\end{array}\right.\]
where $[G,G]$ is the closed subgroup of $G$ generated by the commutators of $G$,
and $G^p$ is the closed subgroup of $G$ generated by the $p$-powers of the elements of $G$.

Let $\lambda_i(G)$ be the elements of the lower $p$-descending central series of the pro-$p$ group $G$,
namely $\lambda_1(G)=G$ and $\lambda_{i+1}(G)=\lambda_i(G)^p[\lambda_i(G),G]$.
In particular, $\lambda_2(G)$ is the Frattini subgroup $\Phi(G)$.
Then, a pro-$p$ group $G$ is called {\it uniformly powerful}, or simply {\it uniform},
if $G$ is finitley generated, powerful, and
\[\left|\lambda_i(G):\lambda_{i+1}(G)\right|=|G:\Phi(G)|\quad\text{for all }i\geq1.\]
Thus a finitely generated powerful group is uniform if, and only if, it is torsion-free (see \cite[Theorem 4.5]{analytic}).

Recall that a pro-$p$ group $G$ is called locally powerful
if every finitely generated closed subgroup $K$ of $G$ is powerful.
Moreover, for uniform pro-$p$ groups, one has the following property:

\begin{prop}[\cite{analytic}, Proposition 4.32]\label{proppresuniform}
 Let $G$ be a $d$-generated uniform pro-$p$ group, and let $\{x_1,\ldots,x_d\}$ be a generating set for $G$.
Then $G$ has a presentation $G=\langle x_1,\ldots,x_d|R\rangle$ with relations
\begin{equation}\label{proppresuniform1}
 R=\left\{[x_i,x_j]=x_1^{\lambda_1(i,j)}\cdots x_d^{\lambda_d(i,j)}, 1\leq i<j\leq d\right\},
\end{equation}
and for all $i,j$ one has $\lambda_n(i,j)\in p.\Z_p$ if $p$ is odd, and $\lambda_n(i,j)\in 4.\Z_2$ if $p=2$.
\end{prop}

If $G$ is a uniform pro-$p$ group, then it is possible to associate a $\Z_p$-Lie algebra
$L=\log(G)$ to it (see \cite[\S 4.5]{analytic} and \cite{ilani}),
i.e., $L$ is the $\Z_p$-free module generated by the generators of $G$,
equipped with the sum
\begin{equation}\label{deflie1}
 x+y=\lim_{n\rightarrow\infty}x+_ny,\quad x+_ny=\left(x^{p^n}y^{p^n}\right)^{p^{-n}},
\end{equation}
and the Lie brackets
\begin{equation}\label{deflie2}
(x,y)=\lim_{n\rightarrow\infty}(x,y)_n,\quad(x,y)_n=\left[x^{p^n},y^{p^n}\right]^{p^{-2n}}.                     
\end{equation}

In analogy to pro-$p$ groups we say that a $\Z_p$-Lie algebra $L$ is {\it powerful} if $L\cong\Z_p^d$
for some $d>0$ as $\Z_p$-module, and the derived algebra $(LL)$ is contained in $p.L$ (resp. in $4.L$ if $p=2$).
It is well known that for a uniform group $G$ the Lie algebra $\log(G)$ is powerful.

\begin{rem}\label{remarklocpwfetLie}
\begin{enumerate}
\item[(i)] If $G=\langle x_1,\ldots,x_n\rangle$ is uniform, then it is possible to write every element $g\in G$
as $g=x_1^{\lambda_1}\cdots x_n^{\lambda_n}$, with $\lambda_i\in\Z_p$, in a unique way.
Thus the map
\[G\longrightarrow\log(G),\quad x_1^{\lambda_1}\cdots x_n^{\lambda_n}\longmapsto\lambda_1.x_1+\ldots+\lambda_n.x_n\]
is a homeomorphism (in the $\Z_p$-topology) \cite[Theorem 4.9]{analytic}.
\item[(ii)] If $G$ is locally powerful and torsion-free, then every closed subfroup $K$ of $G$ is again a uniform group.
Thus one can construct the Lie algebra $\log(K)$, which is in fact a subalgebra of $\log(G)$.
In particular, the $\Z_p$-submodule $\spa_{\Z_p}\{x\in \Omega\}$ of $\log(G)$ is closed under Lie brackets
for every subset $\Omega\subseteq G$.
\end{enumerate}
\end{rem}

\subsection{Oriented pro-$p$ groups}
Let $(G,\theta)$, $\theta\colon G\rightarrow\Z_p^\times$, be an oriented pro-$p$ group.
For every closed subgroup $K\clsgp G$, $(K,\theta|_K)$ is again an oriented pro-$p$.
Notice that the image of $\theta$ is a pro-$p$ subgroup of $\Z_p^\times$,
thus $\image(\theta)\subseteq1+p.\Z_p$.

The following fact is straightforward.

\begin{fact}
 Let $G$ be a pro-$p$ group, and let $\theta,\theta'\colon G\rightarrow \Z_p^\times$, $\theta\neq\theta'$, be two distinct
continuous homomorphisms such that $G$ is both $\theta$- and $\theta'$-abelian.
Then $G$ is a closed subgroup of $\Z_p^\times$.
\end{fact}

The following property will turn out to be useful for our purpose.

\begin{prop}\label{facthetabelian}
Let $(G,\theta)$ be an oriented $d$-generated pro-$p$ group $G$.
Suppose further that $\image(\theta)\clsgp1+4.\Z_2$ if $p=2$.
Then $G$ is $\theta$-abelian if, and only if, there exists a presentation
\begin{equation}\label{presentationthetabelian}
G=\left\langle x_1,\ldots,x_d\left|[x_1,x_i]=x_i^\lambda,[x_i,x_j]=1,2\leq i,j,\leq d\right.\right\rangle,
\end{equation}
where $\lambda=\theta(x_1)-1$ (if $p=2$ then $\lambda\in4.\Z_2$).
\end{prop}

\begin{proof}
Let $G$ be $\theta$-abelian, and put $A=\image(\theta)\clsgp1+p.\Z_p$.
By hypotesis, $A$ is cyclic and torsion-free, i.e., either $A\cong\Z_p$ or $A=1$.
In the latter case $G=\Zen_\theta(G)$, namely, $G$ is abelian.
Otherwise one has the short exact sequence
\begin{displaymath}\xymatrix{1\ar[r] & \Zen_\theta(G)\ar[r] & G \ar[r]^\theta & A \ar[r] & 1,}\end{displaymath}
which splits since $\Z_p$ is a projective pro-$p$ group.
This implies that $G\cong A\ltimes\Zen_\theta(G)$, where the action of $A$ on $\Zen_\theta(G)$ is induced by $\theta$.
Therefore $A\ltimes\Zen_\theta(G)$ has a presentation (\ref{presentationthetabelian}),
where $d=d(\Zen_\theta(G))+1$.

Conversely, suppose $G$ is a pro-$p$ group with presentation (\ref{presentationthetabelian}).
Then one may construct an orientation $\theta\colon G\rightarrow \Z^\times$
such that $\theta(x_1)=1+\lambda$ and $\theta(x_i)=1$ for $i=2,\ldots,d$.
Then $\Zen_\theta(G)$ is generated by $x_2,\ldots,x_d$, and $G$ is $\theta$-abelian.
\end{proof}

\subsection{Oriented $\Z_p$-Lie algebras}
In analogy, we call a $\Z_p$-Lie algebra $L$ together with a continuous homomorphism of Lie algebras
$\theta_L\colon L\to \Z_p$, $\image(\theta_L)\subseteq p.\Z_p$, an {\it oriented $\Z_p$-Lie algebra}.
Thus also in this case one may define the {\it $\theta_L$-center} of $L$ to be the ideal
\[\Zen_{\theta_L}(L)=\left\{v\in\kernel(\theta_L)\;|\;\adj x(v)=\theta_L(x).v\text{ for all }x\in L\right\}.\]
Then $\Zen_{\theta_L}(L)$ is an abelian subalgebra of $L$.
If $\Zen_{\theta_L}(L)=\kernel(\theta_L)$, then we call $L$ a \textit{$\theta_L$-abelian $\Z_p$-Lie algebra}.

The following fact is straightforward.

\begin{fact}\label{facLthetabelian}
A $\Z_p$-Lie algebra $L$ of rank $d$, together with an orientation $\theta_L$, is $\theta_L$-abelian if,
 and only if, $L$ has a basis $\{v_i,\ldots,v_d\}$ such that
 $(v_1,v_i)=\lambda.v_i$ and $(v_i,v_j)=0$ for all $1<i,j\leq d$, where $\lambda=\theta_L(v_1)$
 (if $p=2$ then $\lambda\in4.\Z_2$).
\end{fact}

Combining Proprosition \ref{facthetabelian} and Fact \ref{facLthetabelian}, one obtains the following proposition.
\begin{prop}\label{propthetabelianGL}
A finitely generated uniform pro-$p$ group $G$ with orientation $\theta$ is $\theta$-abelian if,
and only if, the associated Lie algebra $\log(G)$ has an orientation $\theta_L$
such that $\log(G)$ is $\theta_L$-abelian. In particular, $\theta_L=\log(\theta)$ and $\theta=\exp(\theta_L)$.
\end{prop}

\begin{proof}
From the construction of the Lie algebra $\log(G)$ given by (\ref{deflie1}) and (\ref{deflie2}),
and from the presentation (\ref{presentationthetabelian}),
computations show that if $G$ is a uniform $\theta$-abelian pro-$p$ group
then $\log(G)$ has Lie brackets as in Fact \ref{facLthetabelian}.

The map $\log$ from the category of uniform pro-$p$ groups
to the category of powerful Lie algebras over $\Z_p$ is a functor of categories.
Moreover, the group structure can be reconstructed from the Lie algebra structure
by the well known Baker-Campbell-Hausdorff series.
Thus one has the functor $\exp$ from the category of powerful $\Z_p$-Lie algebras 
to the category of uniform pro-$p$ groups, which is the inverse of $\log$.
Namely $\log$ and $\exp$ are mutually inverse isomorphisms between the two categories \cite[Theorem 9.10]{analytic}.

In particular, one has the following commutative diagram:
\begin{displaymath}
 \xymatrix{G\ar@/^/[d]^\log\ar[rr]^\theta && \Z_p^\times\ar@/^/[d]^\log\\
L\ar@/^/[u]^\exp\ar[rr]_{\theta_L} && p.\Z_p\ar@/^/[u]^\exp}
\end{displaymath}
this yields the claim.
\end{proof}

\subsection{Proof of Theorem A}\label{subsec:proofA}

\begin{thmA}\label{ThmA}
A finitely generated uniform pro-$p$ group $G$ is locally powerful
if, and only if, $G$ there exists an orientation $\theta\colon G\rightarrow\Z_p^\times$
such that $(G,\theta)$ is $\theta$-abelian.
\end{thmA}

\begin{proof}
If $G$ is $\theta$-abelian,
then, by Proposition \ref{facthetabelian}, $G$ is locally powerful and torsion-free.

Conversely, let $G$ be a torsion-free locally powerful pro-$p$ group with $d(G)=d\geq2$,
Thus by Proposition \ref{proppresuniform}, $G$ has a presentation $G=\langle x_1,\ldots,x_d|R\rangle$
with relations as in (\ref{proppresuniform1}).
Let $H_{ij}\clsgp G$ be the closed subgroup generated by the elements $x_i,x_j$.
Since $H_{ij}$ is uniform as well, we have that
\[H_{ij}=\left\langle x_i,x_j\left|[x_i,x_j]=x_i^{\lambda_i(i,j)} x_j^{\lambda_j(i,j)},\lambda_i,\lambda_j\in p.\Z_p\right.\right\rangle,\]
so that $R=\{[x_i,x_j]=x_i^{\lambda_i(i,j)} x_j^{\lambda_j(i,j)}, 1\leq i<j\leq d\}$
is the set of relations.

Since an abelian pro-$p$ group is ${\mathbf 1}$-abelian, where ${\mathbf 1}$ is the trivial orientation,
we may assume that $G$ is not abelian, i.e., we may assume without loss of generality that $x_1$ and $x_2$ do not commute.

\medskip
\textit{Step 1:} 
First suppose that $d=2$.
It is well known that if $G$ is nonabelian,
then $G$ has a presentation $\langle x,y|[x,y]y^{-p^k}\rangle$
for some uniquely determined positive integer $k$ \cite[Chapter 4, Exercise 13]{analytic}.
Hence the claim follows from Fact \ref{facthetabelian}.

\medskip
\textit{Step 2:} Suppose $d=3$.
By the previously mentioned remark we may choose $x_1,x_2$ such that $[x_1,x_2]=x_2^\lambda$,
with $\lambda\in p.\Z_p$ (resp. $\lambda\in4.\Z_2$ if $p=2$).
Thus
\[G=\left\langle x_1,x_2,x_3\left| \left[x_1,x_2\right]=x_2^\lambda,\left[x_1,x_3\right]=x_1^{\lambda_1}x_3^{\lambda_2},
\left[x_2,x_3\right]=x_2^{\mu_1}x_3^{\mu_2}\right.\right\rangle,\]
with $\lambda_i,\mu_i\in p.\Z_p$ (resp. in $4.\Z_2$).
Let $H_{ij}$ be the subgroups as defined above, with $1\leq i<j\leq3$, and let $L=\log(G)$.
Clearly, $(x_i,x_j)_n\in H_{ij}$ for all $n$.
Hence $(x_i,x_j)\in\spa_{\Z_p}\{x_i,x_j\}$.
In particular, the Lie brackets in $L$ are such that
\[(x_1,x_2)=\alpha.x_2,\quad (x_2,x_3)=\beta_2.x_2+\beta_3.x_3,\quad (x_1,x_3)=\gamma_1.x_1+\gamma_3.x_3,\]
with $\alpha, \beta_i, \gamma_i\in p.\Z_p$ (resp. in $4.\Z_2$).

By the Jacobi identity, one has
\begin{eqnarray*}
 0 &=& \left((x_1,x_2),x_3\right)+\left((x_2,x_3),x_1\right)+\left((x_3,x_1),x_2\right)\\
   &=& (\alpha.x_2,x_3)+(\beta_2.x_2+\beta_3.x_3,x_1)-(\gamma_1.x_1+\gamma_3.x_3,x_2)\\
   &=& -\beta_3\gamma_1.x_1+\left(\alpha\beta_2-\alpha\beta_2-\alpha\gamma_1+\beta_2\gamma_3\right).x_2
        +\left(\alpha\beta_3-\beta_3\gamma_3+\beta_3\gamma_3\right).x_3,
\end{eqnarray*}
hence $\beta_3\gamma_1.x_1=0$, and thus $\beta_3=0$ or $\gamma_1=0$.
\begin{enumerate}
 \item[(1)] If $\beta_3=0$, then by definition $(x_2,x_3)\in\spa_{\Z_p}\{x_2\}$,
 i.e., $\spa_{\Z_p}\{x_2\}$ is an ideal of $L$.
Therefore we may choose without loss of generality $x_1$ and $x_3$ such that $(x_1,x_3)\in\spa_{\Z_p}\{x_3\}$,
and $(x_i,x_2)\in\spa_{\Z_p}\{x_2\}$ for $i=1,3$.
 \item[(2)] If $\gamma_1=0$, then by definition $(x_1,x_3)\in\spa_{\Z_p}\{x_3\}$,
i.e., $\spa_{\Z_p}\{x_2,x_3\}$ is an ideal of $L$.
Therefore we may choose without loss of generality $x_2$ and $x_3$ such that $(x_2,x_3)\in\spa_{\Z_p}\{x_2\}$,
and $(x_1,x_i)\in\spa_{\Z_p}\{x_i\}$ for $i=2,3$.
\end{enumerate}
Altogether the Lie brackets in $L$ are
\[(x_1,x_2)=\alpha'.x_2,\quad(x_2,x_3)=\beta'.x_2,\quad(x_1,x_3)=\gamma'.x_3,\]
with $\alpha',\beta',\gamma'\in p.\Z_p$ (resp. in $4.\Z_2$).
The matrix of $\adj(\gamma.x_3)$ with respect to the basis $\{x_1,x_2,x_3\}$ is given by
\[\adj(\gamma'.x_3)=\left(\begin{array}{ccc}
   0 & 0 & 0 \\ 0 & \beta'\gamma' & 0 \\ -\gamma'^2 & 0 & 0 \end{array}\right).\]
In particular, its trace is $\tra(\adj(\gamma'.x_3))=\beta'\gamma'$.
Since $\adj(\gamma'.x_3)=(\adj(x_1),\adj(x_3))$,
one has $\tra(\adj(\gamma'.x_3))=\beta'\gamma'=0$.
Therefore $\beta'=0$ or $\gamma'=0$.

\begin{enumerate}
 \item[(1)] If $\beta'=0$, let $v_1=x_1+x_2$ and $v_2=x_2+x_3$.
Then $(v_1,v_2)=\alpha'.x_2+\gamma'.x_3$.
By Remark \ref{remarklocpwfetLie}, one has that $(v_1,v_2)\in\spa_{\Z_p}\{v_1,v_2\}$. 
Thus $(v_1,v_2)$ is necessarily a multiple of $v_2$, i.e., $\alpha'=\gamma'$.
 \item[(2)] If $\gamma'=0$ and $\beta'\neq0$, let $v=x_1+x_2$. Then $(v,x_3)=\beta'.x_2$.
By Remark \ref{remarklocpwfetLie}, one has that $(v,x_3)\in\spa_{\Z_p}\{v,x_3\}$.
In particular, no multiple of $x_2$ lies in $\spa_{\Z_p}\{v,x_3\}$.
Therefore, this case is impossible.
 \item[(3)] If $\beta'=\gamma'=0$ then $\alpha'=0$ by (1). So $L$, and hence $G$, is abelian. But this case was excluded.
\end{enumerate}

This yields $\beta'=0$ and $\alpha'=\gamma'\neq0$, with $\alpha'\in p.\Z_p$ (resp. in $4.\Z_2$).
Therefore, by Fact \ref{facthetabelian} (ii), $L$ is $\theta_L$-abelian, with $\theta_L(x_1)=\alpha'$, $\theta_L(x_i)=0$
for $i=2,3$, and the claim follows from Proposition \ref{propthetabelianGL}.

\medskip

\textit{Step 3:} Finally, suppose that $G$ is locally powerful, torsion-free with $d(G)=n+1\geq4$,
and let $G$ be generated by $x_1,\ldots,x_{n+1}$.
Since $G$ is non-abelian we may assume without loss of generality that $x_1$ and $x_2$ do not commute.

Let $H\clsgp G$ be the subgroup generated by $x_1,\ldots,x_n$.
Thus by induction there is a unique (non-trivial) orientation $\theta\colon H\rightarrow\Z_p^\times$ such that $H$ is $\theta$-abelian.
In particular, we may assume that $[x_1,x_i]=x_i^\lambda$ and $[x_i,x_j]=1$ for all $2\leq i,j\leq n$,
where $\lambda=\theta(x_1)-1\in p.\Z_p\smallsetminus\{0\}$ (resp. in $4.\Z_2\smallsetminus\{0\}$ for $p=2$).

Furthermore, let $H_i\clsgp G$ be the subgroup generated by $x_1,x_i,x_{n+1}$, for $2\leq i\leq n$.
By induction, for each $i$ there exists an orientation $\theta_i\colon H_i\rightarrow\Z_p^\times$ such that $H_i$ is $\theta_i$-abelian.

Since $\theta_i(x_1)=\theta(x_1)=1+\lambda$ and $\theta_i(x_i)=\theta(x_i)=1$ for all $i$,
then necessarily $\theta_i(x_{n+1})=1$ for all $i$;
i.e., $[x_1,x_{n+1}]=x_{n+1}^\lambda$ and $[x_i,x_{n+1}]=1$ for all $i$.
Hence we may extend $\theta$ to $G$ such that $\theta(x_{n+1})=1$.
Thus $G$ is $\theta$-abelian.

This estabilishes the theorem.\end{proof}

\section{A Tits alternative for Bloch-Kato pro-$p$ groups}\label{sec:titsalt}

\subsection{Dimension of cohomology groups}
If $p=2$ then the cohomology ring of a Bloch-Kato group $G$ is a quotient of the symmetric algebra $S^\bullet(H^1(G))$.
On the other hand, if $p$ is odd then
the cohomology ring of $G$ is a quotient of the exterior algebra $\bigwedge_\bullet(H^1(G))$. 
Thus in this latter case if $G$ is finitely generated then
\begin{equation}\label{inequalityBloch-Kato}
 \dim_{\F_p}(H^r(G))\leq\binom{d(G)}{r}\quad\text{for all }r\geq0.
\end{equation}
In fact it is possible to prove a stronger result. 

\begin{prop}\label{propdimwBloch-Kato}
Let $p$ be odd, and let $G$ be a finitely generated Bloch-Kato pro-$p$ group.
Then
\begin{itemize}
 \item[(i)] $cd(G)\leq d(G)\leq\dim_{\F_p}\left(\lambda_2(G)/\lambda_3(G)\right)$;
 \item[(ii)] $r(G)\leq\binom{d(G)}{2}$.
\end{itemize}
\end{prop}

\begin{proof}
The inequalities $cd(G)\leq d(G)$ and $r(G)\leq\binom{d(G)}{2}$ are immediate consequences of (\ref{inequalityBloch-Kato}).

The inflation map induces an isomorphism $\rho=\inf_{G/\Phi(G)}^1$ in degree 1,
so that the commutativity of the diagram
\begin{displaymath}
 \xymatrix{H^1(G/\Phi(G))\otimes H^1(G/\Phi(G))\ar[d]^{\rho\otimes\rho}_\wr\ar@{->}[rr]^-\cup && H^2(G/\Phi(G))\ar@{->}[d]^{\inf_{G/\Phi(G)}^2}\\
H^1(G)\otimes H^1(G)\ar@{->>}[rr]^\cup && H^{2}(G)}
\end{displaymath}
implies that $\inf_{G/\Phi(G)}^2$ is surjective.

Consider the five terms exact sequence arising from the quotient $G/\Phi(G)$.
Since $\rho$ is an isomorphism, it reduces to
\begin{equation}\label{5tes}
\begin{CD}
0 @>>> H^1(\Phi(G))^G @>>> H^2(G/\Phi(G)) @>{\inf_{G/\Phi(G)}^2}>> H^2(G) @>>> 0.
\end{CD} 
\end{equation}
Moreover, the group $H^1(\Phi(G))^G$ is isomorphic to the quotient $(\lambda_2(G)/\lambda_3(G))^*$ as discrete group,
where $\_^*$ denotes the Pontryagin dual.

Since $G/\Phi(G)$ is a elementary abelian $p$-group, the second cohomology group is
\begin{eqnarray*}
H^2(G/\Phi(G))&=&\beta\left(H^1(G/\Phi(G))\right)\oplus\left(H^1(G/\Phi(G))\cup H^1(G/\Phi(G))\right)\\
&\cong&H^1(G/\Phi(G))\oplus\left(H^1(G/\Phi(G))\wedge H^1(G/\Phi(G))\right)\\
&\cong&H^1(G)\oplus\left(H^1(G)\wedge H^1(G)\right).
\end{eqnarray*}

From the sequence (\ref{5tes}) one obtains 
\begin{equation}\label{dimensionBlochKato}
0\longrightarrow \left(\lambda_2(G)/\lambda_3(G)\right)^*\longrightarrow H^1(G)\oplus\left(H^1(G)\wedge H^1(G)\right)
\longrightarrow H^2(G)\longrightarrow 0.
\end{equation}
Therefore
\begin{eqnarray*}
 d(G)+\binom{d(G)}{2} &=& \dim_{\F_p}\left(H^1(G)\oplus \left(H^1(G)\wedge H^1(G)\right)\right)\\
 &=&\dim_{\F_p}\left(\frac{\lambda_2(G)}{\lambda_3(G)}\right)+\dim_{\F_p}\left(H^2(G)\right)\quad\text{by (\ref{dimensionBlochKato})}\\
 &=&\dim_{\F_p}\left(\frac{\lambda_2(G)}{\lambda_3(G)}\right)+r(G)\\
 &\leq&\dim_{\F_p}\left(\frac{\lambda_2(G)}{\lambda_3(G)}\right)+\binom{d(G)}{2},
\end{eqnarray*}
namely $d(G)\leq\dim_{\F_p}(\lambda_2(G)/\lambda_3(G))$.
\end{proof}

\begin{rem}
There is no analogue of Proposition \ref{propdimwBloch-Kato} in case that $p=2$. For a Bloch-Kato pro-2 group $G$ the
exact sequence (\ref{5tes}) specifies to
\[0\longrightarrow \left(\lambda_2(G)/\lambda_3(G)\right)^*\longrightarrow H^2\left((\Z/2\Z)^d\right)\longrightarrow H^2(G)\longrightarrow 0,\]
and $\dim(H^2((\Z/2\Z)^d))=\binom{d+1}{2}$, while $\dim(H^2(G))\leq\binom{d+1}{2}$.
\end{rem}

\begin{prop}\label{cd(G)=d(G)}
Let $p$ be odd, and let $G$ be a Bloch-Kato pro-$p$ group
such that $cd(G)=d(G)$.
Then the cohomology ring $H^\bullet(G)$ is isomorphic to the $\F_p$-exterior algebra $\bigwedge_\bullet(H^1(G))$. 
\end{prop}

\begin{proof}
Let $H^1(G)$ be freely generated by $\chi_1,\ldots,\chi_d$ as $\F_p$ vector space,
and suppose for contradiction that $H^\bullet(G)$ is a non-trivial quotient of $\bigwedge_\bullet(H^1(G))$.
Since $H^\bullet(G)$ is quadratic, there is a non-trivial relation in $H^1(G)\wedge H^1(G)$.
Thus we may assume without loss of generality that
\[\chi_1\cup\chi_2=\sum_{(i,j)\neq(1,2)}a_{ij}.\chi_i\cup\chi_j,\]
with $i<j$ and $a_{ij}\in \F_p$.
This implies that 
\[\chi_1\cup\chi_2\cup\cdots\cup\chi_d=\sum_{(i,j)\neq(1,2)}a_{ij}.\chi_i\cup\chi_j\cup\chi_3\cup\cdots\cup\chi_d=0,\]
namely $H^d(G)=\spa_{\F_p}\{\chi_1\cup\cdots\cup\chi_d\}=0$, a contradiction.
This yields the claim.
\end{proof}

\subsection{Powerful groups and the cup product}
The following theorem is due to P. Symonds and Th. Weigel:

\begin{thm}[\cite{symondsthomas}, Theorem 5.1.6]\label{thmsymondsthomas}
Let $G$ be a finitely generated pro-$p$ group.
Then the map
\[\Lambda_2(\cup)\colon H^1(G)\wedge H^1(G)\longrightarrow H^2(G)\]
induced by the cup product is injective if, and only if, $G$ is powerful.
\end{thm}

Let $G$ be a pro-$p$ group, and let $H$ be a closed subgroup of $G$.
Then we call $H$ {\it properly embedded} in $G$, if the canonical map
$H/\Phi(H)\to G/\Phi(G)$ is injective. The following fact is a direct consequence
of Pontryagin duality.

\begin{fact}
\label{fact:proper1}
Let $G$ be a pro-$p$ group, and let $H$ be a closed subgroup of $G$.
Then the following are equivalent.
\begin{itemize}
\item[(i)] $H$ is properly embedded in $G$.
\item[(ii)] $\res^1_{G,H}\colon H^1(G,\F_p)\to H^1(H,\F_p)$ is surjective.
\end{itemize}
\end{fact}

\begin{thm}\label{thmB}
\label{thm3equivalenceBloch-KatoHP}
Let $p$ be an odd prime, and let $G$ be a Bloch-Kato pro-$p$ group. 
Then the following are equivalent:
\begin{itemize}
\item[(i)] $G$ does not  contain non-abelian closed free pro-$p$ subgroups.
\item[(ii)] $G$ is locally powerful.
\item[(iii)] there exists an orientation $\theta\colon G\rightarrow \Z_p^\times$ such that $(G,\theta)$ is $\theta$-abelian.
In particular, $G$ is metabelian.
\end{itemize}
Moreover, if $G$ is finitely generated, then \textup{(i)}, \textup{(ii)}, and \textup{(iii)} are equivalent to:
\begin{itemize}
\item[(iv)] $G$ is $p$-adic analytic.
\end{itemize}
\end{thm}

\begin{proof}
Suppose that (i) holds and that $G$ is not locally powerful. 
Then there exists a finitely generated subgroup $K\clsgp G$ which is not powerful.
In particular, the map
\begin{equation*}
\Lambda_2(\cup)\colon H^1(K)\wedge H^1(K)\longrightarrow  H^2(K)
\end{equation*}
is not injective. 
Let $\chi_1,\ldots,\chi_r$ be an $\F_p$-basis of the $\F_p$-vector space $H^1(K)$.
Thus there exists a non-trivial element
\[\eta=\sum_{1\leq i<j\leq r}\,a_{ij}.\chi_i\wedge\chi_j\in\kernel\left(\Lambda_2(\cup)\right).\]
As $\eta\not=0$, there exist $m,n\in\{1,\ldots,r\}$, $m<n$, such that $a_{mn}\not=0$.
Let $x_1,\ldots,x_r\in K$ be a minimal generating system of $K$
satisfying $\chi_i(x_j)=\delta_{ij}$ for all $i,j\in \{1,\ldots,r\}$,
and let $S=\langle x_m,x_n\rangle$. Then $S$ is properly embedded in $K$,
$\rho=\res^1_{K,S}\colon H^1(K)\to H^1(S)$ is surjective, and, by construction,
$\kernel(\rho)=\spa_{\F_p}\{\,\chi_i\mid 1\leq i\leq r,\, i\not=n,m\,\}$. 
From the surjectivity of $\rho\wedge\rho$ and the commutativity of the diagram
\begin{displaymath}
 \xymatrix{H^1(K)\wedge H^1(K)\ar@{->>}[d]_{\rho\wedge\rho}\ar[rr]^-{\Lambda_2(\cup)} && H^2(K)\ar@{->}[d]^{\res^2_{K,S}}\\
H^1(S)\wedge H^1(S)\ar[rr]^-{\Lambda_2(\cup)} && H^{2}(S)}
\end{displaymath}
one concludes that the map $\Lambda_2(\cup)\colon H^1(S)\wedge H^1(S)\to H^2(S)$ is the $0$-map.
Thus -- as $S$ is Bloch-Kato -- $H^2(S)=0$, i.e., $S$ is a 2-generated free pro-$p$ group
(see \cite[Proposition 3.5.17]{nsw}), a contradiction.
This shows that (i) implies (ii).

The implication $\text{(ii)}\Rightarrow\text{(i)}$ follows from the fact
that a free pro-$p$ group which is powerful must be cyclic.
Moreover, the equivalence $\text{(ii)}\Leftrightarrow \text{(iii)}$ follows from Theorem A.
If $G$ is finitely generated, the implication $\text{(ii)}\Rightarrow \text{(iv)}$ is well known 
(see \cite[Theorem 8.18]{analytic}),
whereas the implication $\text{(iv)}\Rightarrow \text{(i)}$ follows from \cite[Theorem 8.32]{analytic}.
This yields the claim.
\end{proof}

\begin{rem}
Notice that in the proof we do not require the group $G$ to be Bloch-Kato;
in fact it is enough to assume that the cohomology of every closed subgroup of $G$ is 
decomposable, i.e., it is generated in degree one (thus $G$ is {\it almost} Bloch-Kato, in the language of \cite{bk}). 
\end{rem}

\begin{cor}\label{corollmetabelian}Let $p$ be an odd prime, and let $G$ be a Bloch-Kato pro-$p$ group.
Then the following are equivalent
\begin{itemize}
\item[(i)] $G$ is powerful.
\item[(ii)] $G$ contains no free pro-$p$ groups of infinite rank.
\item[(iii)] There exists an orientation $\theta:G\rightarrow\Z_p^\times$ such that $G$ is $\theta$-abelian.
\end{itemize}
Furthermore, if $G$ is finitley generated, these properties are equivalent to
\begin{itemize}
\item[(iv)] $G$ is $p$-adic analytic.
\item[(v)] $cd(G)=d(G)$.
\item[(vi)] $H^\bullet(G)\cong\bigwedge_\bullet\left((G/\Phi(G))^*\right)$.
\end{itemize}
\end{cor}

As we stressed in \S \ref{sec:prelim}, Bloch-Kato groups arise naturally as maximal pro-$p$ Galois groups
and $p$-Sylow subgroups of absolute Galois groups.
Thus the above results provide strong restrictions to such groups. In particular, one obtains the following result:

\begin{cor}\label{corollware}
Let $F$ be a field, such that $G_F(p)$ is a metabelian  pro-$p$ group (i.e., the commutator subgroup of $G_F(p)$ is abelian).
If $F\supseteq\mu_p$ then $G_F(p)$ has generators $\{\sigma,\rho_i\}_{i\in\mathcal{I}}$
with relations $[\rho_i,\rho_j]=1$ and $\rho_i^\sigma=\rho_i^{q+1}$,
where $q=0$ if $\mu_{p^k}\subseteq F$ for all $k\geq1$,
or $q=p^n$, where $n$ is the largest integer such that $F\supseteq\mu_{p^n}$.
\end{cor}

This corollary provides the answer to a question raised by R.~Ware in his paper \cite[page 727]{ware1}.
Indeed he managed to prove that $G_F(p)$ has such a presentation 
if $F$ contains also a $p^2$th root of unity (and not only a $p$th root),
though it seemed reasonable that such an assumption is not necessary -- as, in fact, it is not.

In this case, the suitable orientation $\theta$ of $G_F(p)$ is the cyclotomic character, i.e., the map
\[\theta\colon G_F(p)\longrightarrow \Endm_F(\mu_{p^\infty})\cong\Z_p^\times,\]
where $\mu_{p^\infty}\leq \bar{F}^{sep}$ denotes the group of roots of unity of $p$-power order.

In particular, the $\theta$-center is $\Zen_\theta(G_F(p))=G_L(p)$, where $L=F(\mu_{p^\infty})$.

\begin{example}\label{examplesemidirect}
Let $q=p^n$ be a (non-trivial) $p$-power,
and let $F$ be the field $F=k((\mathfrak{X}))$,
where $k=\F_\ell(\mu_q)$, with $\ell\equiv1\bmod p$, and $\mathfrak{X}=\{X_1,\ldots,X_n\}$.
Then $G_F(p)$ has generators $\{\sigma,\rho_i\}_{i=1}^n$
with relations $[\rho_i,\rho_j]=1$ and $\rho_i^\sigma=\rho_i^{q+1}$.
Furhtermore, if $\mu_q\subseteq k$ for every $p$-power $q$,
then $G_F(p)$ is abelian, i.e., $G_F(p)\cong\Z_p^n$.
\end{example}

\bigskip

The case $p=2$ is more subtle, since the pro-2 version for Theorem \ref{thmsymondsthomas} is more involuted.
Thus it turns out that it is impossible to state Theorem~B also for Bloch-Kato pro-2.
For example the pro-2 dihedral group 
\[C_2\ltimes\Z_2(2)=\left\langle\sigma,\rho\left|\sigma^2=1,{^\sigma\rho}=\rho^{-1}\right.\right\rangle\]
is $\theta$-abelian, with $\theta(\sigma)=-1$, $\theta(\rho)=1$,
and it contains no non-abelian closed free pro-2 subgroups, yet it is not powerful.

Nevertheless, it is possible to get a similar result when we add more restrictions to $G$,
and using \cite[Theorem C]{weigelcohom}.

\begin{thm}
 Let $G$ be a Bloch-Kato pro-2 group such that $G$ is torsion-free,
and assume that the first Bockstein homomorphism $\beta:H^1(G)\rightarrow H^2(G)$ is trivial.
Then the following are equivalent:
\begin{itemize}
\item[(i)] Every non-trivial closed free subgroup of $G$ is cyclic.
\item[(ii)] $G$ is locally powerful.
\item[(iii)] there exists an orientation $\theta\colon G\rightarrow \Z_2^\times$ such that $(G,\theta)$ is $\theta$-abelian.
\end{itemize}
\end{thm}

\section{The class of Bloch-Kato pro-$p$ groups}

Some time ago I. Efrat has formulated a conjecture -- the so called ``elementary type conjecture'' --
for maximal pro-$p$ Galois groups,
which states that the group structure of maximal pro-$p$ Galois groups of some fields is very restricted,
namely such groups are free pro-$p$ products and semidirect products
of certain pro-$p$ groups (see \cite{eltypconj}, \cite{eltypconj2}).

It seems very difficult to decide whether such an ``elementary type'' conjecture should hold 
already for the class of finitely generated Bloch-Kato pro-$p$ groups.
All known examples of Bloch-Kato pro-$p$ groups have this property, but apart from this fact there is little evidence.

For this reason we investigate certain closure operations for the class of Bloch-Kato pro-$p$ groups. 

\subsection{Projective limits and free products of Bloch-Kato groups}
\begin{prop}\label{propinverselimit}
Let $\{G_i,\pi_{ij}\}_{i\in I}$ be projective system of Bloch-Kato pro-$p$ groups
with $\pi_{ij}$ surjective for all $i\leq j$, such that the maps 
\[\text{\emph{inf}}_{ij}^\bullet:H^\bullet(G_j)\rightarrow H^\bullet(G_i)\]
induced by $\pi_{ij}:G_j\rightarrow G_i$ are injective for any $i\leq j$.
Then for $\hat{G}=\varprojlim_iG_i$, the cohomology ring $H^\bullet(\hat{G})$ is quadratic.
\end{prop}

\begin{proof}
It is well known that 
\[\varinjlim_{i\in I}H^n(G_i)\cong H^n(\hat{G})\]
for every $n\geq0$ \cite[Proposition 1.5.1]{nsw}.

Moreover, the class of quadratic $\F_p$-algebras is closed under certain direct limits:
namely if $A_\bullet^i$ is a quadratic $\F_p$-algebra for all $i\geq0$ with $A_\bullet=\varinjlim_iA_\bullet^i$
and such that the maps $A_n^i\rightarrow A_n^j$ are injective for all $i\leq j$,
then $A_\bullet$ is quadratic.
This implies that $H^\bullet(\hat{G})$ is quadratic.
\end{proof}

In order to state and prove the following theorem, we need O. Mel'nikov's version of the Kurosh subgroup theorem
for free pro-$p$ products (see \cite{freeprodMelnik}).

Let $T$ be a profinite space, and let $\{G_t\}_{t\in T}$ be a family of pro-$p$ groups.
Then such a family defines a {\it sheaf} $\mathcal{G}$ of pro-$p$ groups, i.e., a profinite space $\mathcal{G}$ together
with a continuous surjection $\gamma:\mathcal{G}\rightarrow T$ such that for all $t\in T$, $\gamma^{-1}(t)=G_t$,
and the group operation of $G_t$ depends continuously on $t$.
The free pro-$p$ product of the family $\{G_t\}$ is the pro-$p$ group $G=\coprod_t G_t$ together with a morphism
$\iota\colon\mathcal{G}\rightarrow G$ such that for any pro-$p$ group $H$ and for any
continuous map $\varphi\colon\mathcal{G}\rightarrow H$ whose restrictions $\varphi|_{G_t}\colon G_t\rightarrow H$
are all homomorphisms of pro-$p$ groups, there exists a unique homomorphism $\widetilde{\varphi}\colon G\rightarrow H$
such that $\widetilde{\varphi}\circ\iota=\varphi$.

\begin{thm}\label{freeprodofBloch-Kato}
Let $G=\coprod_tG_t$ be the free product in the category of pro-$p$ groups of a family of Bloch-Kato pro-$p$ groups
$\{G_t\}_{t\in T}$, where $T$ is a profinite space.
Then $G$ is a Bloch-Kato pro-$p$ group.
In particular, the free pro-$p$ product of two Bloch-Kato pro-$p$ groups $G_1$ and $G_2$ is a Bloch-Kato pro-$p$ group.
\end{thm}

\begin{proof} 
Let $K$ be a closed subgroup of $G$.
Then by \cite[Theorem 4.3]{freeprodMelnik} it is possible to decompose $K$ in the following way:
\[K=\left(\coprod_{t\in T}\left(\coprod_{ K\backslash G/G_t}\left(K\cap G_t^r\right)\right)\right)\sqcup S,\]
where $S$ is a free pro-$p$ group and the $r$ vary over a set $\mathcal{R}_t\subset G$ of representatives of the coset
space $K\backslash G/G_t$ -- which is profinite.

In particular, $K$ is the free pro-$p$ product (over a profinite set) of closed subgroups of the groups $G_t$.
Let $K_t=\coprod_{r\in\mathcal{R}_t}(K\cap G_t^r)$, so that $K=S\sqcup(\coprod_tK_t)$.
As a consequence of \cite[Theorems 4.1 and 4.2]{freeprodMelnik}, one has that the homology corestriction maps
\begin{eqnarray*}
&&\text{cor}_n^{K_t}\colon\bigoplus_{K\backslash G/G_t}H_n\left(K\cap G_t^r,\F_p\right)\longrightarrow H_n\left(K_t,\F_p\right),\quad\text{and}\\
&&\text{cor}_n^{K}\colon\bigoplus_{t\in T}H_n\left(K_t,\F_p\right)\longrightarrow H_n\left(K,\F_p\right)
\end{eqnarray*}
are isomorphisms for $n\geq1$.
By Pontryagin duality, i.e., \[H_\bullet\left(G,\F_p\right)^*\cong H^\bullet\left(G,\F_p^*\right),\]
also the cohomology restriction maps
\begin{eqnarray}\label{freecohom1}
&&\text{res}_{K_t}^n\colon H^n\left(K_t\right)\longrightarrow\bigoplus_{K\backslash G/G_t}H^n\left(K\cap G_t^r\right),\quad\text{and}\\
&&\text{res}_K^n\colon H^n\left(K\right)\longrightarrow\bigoplus_{t\in T}H^n\left(K_t\right)\label{freecohom2}
\end{eqnarray}
are isomorphisms for $n\geq1$.
Since the groups $G_t$ are Bloch-Kato pro-$p$ groups, the cohomology rings $H^\bullet(K\cap G_t^r)$ are quadratic.
Thus, by (\ref{freecohom1}) and (\ref{freecohom2}) also the cohomology rings $H^\bullet(K_t)$ and $H^\bullet(K)$
are quadratic.
Therefore $G$ is Bloch-Kato.\end{proof}

\subsection{Direct products of Bloch-Kato groups}
It seems natural to consider direct products of Bloch-Kato pro-$p$ groups.
The next result is a consequence of Theorem~A:

\begin{prop}\label{propdirectprodtheta}
The direct product of a powerful non-abelian Bloch-Kato group $G$ with any pro-$p$ group is not Bloch-Kato.
\end{prop}

\begin{proof}
Let $\widetilde{G}=G\times \Z_p$, with $G$ $\theta$-abelian, but not abelian,
and suppose $\widetilde{G}$ is Bloch-Kato.
Since $\widetilde{G}$ contains no free pro-$p$ groups of rank greater than 1,
$\widetilde{G}$ must be $\widetilde{\theta}$-abelian, for some orientation
$\widetilde{\theta}\colon\widetilde{G}\rightarrow\Z_p^\times$ such that $\widetilde{\theta}|_G=\theta$.
But the action of $G$ on $\Z_p$ is trivial, and $G$ is non-abelian (i.e., $\theta$ is not trivial),
thus $\Z_p\nleqslant\Zen_\theta(G)$, and $\widetilde{G}$ cannot be $\widetilde{\theta}$-abelian, a contradiction.
Therefore $\widetilde{G}$ is not Bloch-Kato,
and this implies that the direct product of a powerful non-abelian Bloch-Kato group
with any non-trivial pro-$p$ group is not Bloch-Kato.
\end{proof}

However, one has the following.

\begin{thm}\label{thmdirectfrouctZpfree}
Let $S$ be a free pro-$p$ group, and let $\widetilde{G}=\Z_p\times S$.
Then $\widetilde G$ is a Bloch-Kato pro-$p$ group. 
\end{thm}

\begin{proof}
First of all we show that the cohomology ring $H^\bullet(\widetilde G)$ is a quadratic $\F_p$-algebra.
By \cite[Theorem 2.4.6]{nsw}, one has that 
\[H^n\left(\widetilde G\right)=H^n(S)\oplus H^{n-1}(S),\quad \text{for }n\geq1.\]
In particular, let $\{x_i\}_{i\in I}$ be a set of free generators of $S$, and let $y$ be a generator of $\Z_p$.
Moreover, let $x_i^*\in H^1(S)$ be the Pontryagin dual of $x_i$, for all $i\in I$,
and let $y^*\in H^1(\Z_p)$ be the dual of $y$.
Then 
\[H^\bullet\left(\widetilde G\right)=\F_p\oplus\spa_{\F_p}\left\{ x_i^*,y^*\right\}_{i\in I}\oplus
\spa_{\F_p}\left\{ x_i^*\cup y^*\right\}_{i\in I},\]
namely, $H^\bullet(\widetilde G)$ is a quadratic $\F_p$-algebra.

Let $K$ be a closed subgroup of $\widetilde G$, and put $N=K\cap S$.
Then $N$ is a free pro-$p$ group, and $N\clno K$, moreover one has the following commutative diagram:
\begin{displaymath}
 \xymatrix{
1\ar[r] & S\ar[r] & \widetilde G \ar[r] & \Z_p\ar[r] & 1\\
1\ar[r] & N\ar[r]\ar@{^{(}->}[u] & K \ar[r]\ar@{^{(}->}[u] & A\ar[r]\ar@{^{(}->}[u] & 1}
\end{displaymath}
where either $A$ is isomorphic to $\Z_p$ or it is trivial.
In the latter case $K\cong N$, and $K$ is a Bloch-Kato group.
Otherwise, the lower line of the diagram becomes
\begin{equation}\label{sesdirectproduct}
 1\longrightarrow N\longrightarrow K\longrightarrow \Z_p\longrightarrow 1.
\end{equation}
Since $\Z_p$ is a projective pro-$p$ group, (\ref{sesdirectproduct}) splits,
and we have the isomorphism $K\cong \Z_p\ltimes N$.
In particular, $K\cong\Z_p\times N$, since $A\leq\Zen(\widetilde G)$.
Therefore, the above argument implies that $H^\bullet(K)$ is a quadratic $\F_p$-algebra, and this proves the theorem.
\end{proof}

On the other hand, direct products of non-abelian free pro-$p$ groups are not Bloch-Kato.
For the proof of the next theorem we make use of the following fact, the easy proof of which we leave to the reader.

\begin{fact}\label{modulefact}
Let $G$ be a finite $p$-group, and let $M$ be a finitely generated left $\F_p[G]$-module. Then
\begin{equation}\label{moduleformula}
\dim_{\F_p}\left(M^G\right)\cdot|C|\geq\dim_{\F_p}(M) 
\end{equation}
\end{fact}

\begin{thm}\label{thmdirectproductfree}
Let $F_2$ be a 2-generated free pro-$p$ group. Then $F_2\times F_2$ is not Bloch-Kato.
\end{thm}

\begin{proof}
Let $F_2$ be generated by $x$ and $r$, and consider the presentation
\begin{displaymath}
 \xymatrix{1\ar[r]&R\ar[r]&F_2\ar[r]^\varphi &\Z_p\ar[r]&1,}
\end{displaymath}
where $R$ is generated as closed normal subgroup by $ r $.
We call $\Gamma=F_2\times_\varphi F_2$
the pullback object of the diagram
\begin{displaymath}
\xymatrix{\Gamma\ar[r]\ar[d] & F_2\ar[d]^{\varphi}\\
 F_2\ar[r]_\varphi & \Z_p } 
\end{displaymath}
in the category of pro-$p$ groups. Namely,
\[\Gamma=F_2\times_\varphi F_2=\left\{\left(\xi,\xi'\right)\in F_2\times F_2 \left| \varphi(\xi)=\varphi(\xi')\right.\right\}\clsgp F_2\times F_2.\]
In particular, $\Gamma$ is generated by the pairs $(x,x)$, $(r,r)$ and $(r,1)$.
In fact, every element $(y,y')\in\Gamma$ can be written as 
\begin{equation}\label{remarkGamma}
 (y,y')=(\xi,\xi)(\rho,1)=(\rho',1)(\xi,\xi),
\end{equation}
with $\xi\in F_2$ and $\rho,\rho'\in R$.

Let $R^{\abel}=R/[R,R]$ be the abelianization of $R$, and let $H=R^{\abel}\rtimes \Z_p$,
where the action of $\Z_p$ on $R^{\abel}$ is induced by the conjugation of $F_2$.
Moreover, let $D\clno \Gamma$ be generated as closed normal subgroup of $\Gamma$ by $(r,r)$.

By (\ref{remarkGamma}), the semidirect product $R\rtimes F_2$ maps onto $\Gamma$ via the map $\varphi'$, where
\[\varphi'(\rho,\xi)=(\rho,1)(\xi,\xi),\quad\text{with }\rho\in R,\:\xi\in F_2,\]
which is easily seen to be a homomorphism.
Furthermore, $\varphi'$ is injective, so $R\rtimes F_2$ is isomorphic to $\Gamma$.

Suppose that  $\varphi'(\rho,\xi)=(\rho,1)(\xi,\xi)\in D$.
Thus $\xi\in R$, which implies that $(\xi,\xi),(\rho,1)\in D$.
By (\ref{remarkGamma}), one has that every element of $D$ is generated by elements
$({^\xi\rho},\rho)$, with $\rho,\xi\in R$.
Furthermore,
\[{^{\xi_1}\rho_1}\cdot{^{\xi_2}\rho_2}={^{\xi_1}(\rho_1\rho_2)}y,\quad\text{with }\rho_i,\xi_i\in R,y\in[R,R].\]
Thus, a limit argument shows that $(\rho,1)\in D$ if, and only if, $\rho\in[R,R]$.
This implies that 
\[\frac{\Gamma}{D}\cong H=\frac{R\rtimes F_2}{[R,R]\rtimes R}.\]

We want to show that $\Gamma$ is not finitely presented. 
Assume for contradiction that it is.
Then also $H$ is finitely presented, since $D$ is finitely generated as normal subgroup of $\Gamma$.

\begin{step}
The group $H$ is not finitely presented, i.e., $r(H)=\infty$.
\end{step}
The Hochschild-Lyndon-Serre spectral sequence $H^r(\Z_p,H^s(R^{\abel}))\Rightarrow H^{r+s}(H)$
collapses at the $E_2$-term, i.e., $E_\infty=E_2$, since $cd(\Z_p)=1$.
Thus one has the following exact sequence:
\begin{equation}\label{HLSes}
0\longrightarrow H^1\left(\Z_p,H^1(R^{\abel})\right)\longrightarrow H^2(H)\longrightarrow H^2(R^{\abel})^{\Z_p}\longrightarrow0.
\end{equation}

Since $R$ is one-generated as normal subgroup, the group $R^{\abel}$ is isomorphic to the completed group algebra $\Z_p\dbl \Z_p\dbr$.
Thus the first cohomology group $H^1(R^{\abel})$ is isomorphic to $(\F_p\dbl \Z_p\dbr)^*$, and the second cohomology group
$H^2(R^{\abel})$ is isomorphic to the second exterior algebra $\Lambda_2(\F_p\dbl \Z_p\dbr^*)$.
Since $\F_p\dbl \Z_p\dbr\cong\varprojlim_k\F_p\dbl C_{p^k}\dbr$, where $C_{p^k}$ is the cyclic group of order $p^k$,
one has 
\[H^2(R^{\abel})^\Z_p\cong\varinjlim_{k,U_k}\Lambda_2\left(\F_p\dbl C_{p^k}\dbr^*\right)^{\Z_p/U_k},\]
where $U_k\opno \Z_p$ is such that $\Z_p/U_k\cong C_{p^k}$.

Therefore, Fact \ref{modulefact} implies that for all $k\geq1$ one has
\[\dim_{\F_p}\left(\Lambda_2\left(\F_p\dbl C_{p^k}\dbr^*\right)^{\Z_p/U_k}\right)\geq
\frac{1}{|\Z_p/U_k|}\dim_{\F_p}\Lambda_2\left(\F_p\dbl C_{p^k}\dbr^*\right)=\frac{1}{p^k}\frac{p^k(p^k-1)}{2}.\]
Thus $\dim(H^2(R^{\abel})^{\Z_p})\geq(p^k-1)/2$ for all $k\geq1$, i.e., $H^2(R^{\abel})^{\Z_p}$ has infinite dimension,
and so by (\ref{HLSes}) $\dim(H^2(H))=r(H)=\infty$.
This estabilishes the claim.

\medskip

Therefore the claim implies that $\Gamma$ is not finitely presented.
In particular, $\dim(H^2(\Gamma))=r(\Gamma)=\infty$, whereas $\dim(H^1(\Gamma))=d(\Gamma)=3$.
Hence $H^\bullet(\Gamma)$ is not quadratic, and $F_2\times F_2$ is not Bloch-Kato.
\end{proof}

\begin{rem}
In particular, Theorem~\ref{thmdirectproductfree} shows that $F_2\times F_2$ is not a \textit{coherent} pro-$p$ group,
i.e., it contains a finitely generated group which is not finitely presented.  
By Proposition~\ref{propdimwBloch-Kato}, any finitely generated Bloch-Kato pro-$p$ group is a coherent group.
\end{rem}

Now Theorem~C is the combination of Proposition~\ref{propdirectprodtheta}, Theorem~\ref{thmdirectfrouctZpfree}
and Theorem~\ref{thmdirectproductfree}.

\begin{thmC}
Let $G_1$ and $G_2$ be Bloch-Kato pro-$p$ groups, and assume that $G_1\times G_2$ is Bloch-Kato as well.
Then the following restrictions hold:
\begin{itemize}
 \item[(i)] None of $G_1$ and $G_2$ is a powerful non-abelian Bloch-Kato group;
 \item[(ii)] at least one of the two groups is abelian.
\end{itemize}
In particular, $\Z_p\times S$ is a Bloch-Kato pro-$p$ group for any free pro-$p$ group $S$.
\end{thmC}

\section*{Aknowledgement}
We would like to thank John Labute, Franco Magri, J\'an Min\'a\v{c} and Dan Segal, for encouragement, help and 
many other reasons.
Thanks also to Sunil K. Chebolu, Jochen G\"artner, Hilaf Hasson, Danny Krashen and Danny Neftin for the interest shown
in the topics of this article.
And above all thanks Thomas S. Weigel, whitout the guidance of whom this article would not exist.

\end{document}